\documentclass[11pt, oneside]{article}   	
\usepackage{zgsty}     
\usepackage{authblk}

\title{Multi-Scale Analysis on Complex Networks using Hermitian Graph Wavelets}
\author[1,2]{Zach Gelbaum$^*\,$}
\author[1,3] {Mathew Titus}
\author[1] {James Watson}
\affil[1]{College of Earth, Ocean and Atmospheric Sciences, Oregon State University}
\affil[2]{Northwest Mathematics LLC}
\affil[3]{The Prediction Lab LLC}

\date{}							

\begin{document}
\maketitle
\begin{abstract}
We construct and study a class of spectral graph wavelets by analogy with Hermitian wavelets on the real line.  We provide a localization result that significantly improves upon those previously available, enabling application to highly non-sparse, even complete, weighted graphs.  We then define a new measure of importance of a node within a network called the Maximum Diffusion Time, and conclude by establishing an equivalence between the maximum diffusion time and information centrality, thus suggesting applications to quantifying hierarchical and distributed leadership structures in groups of interacting agents.
\end{abstract}

\section{Introduction}
{\let\thefootnote\relax\footnote{\textsuperscript{*}Corresponding author: zach.gelbaum@northwestmath.com}
With the remarkable success and ubiquity of wavelets as a tool for multiscale analyses in classical settings, a large body of work has been devoted to extending the theory of wavelets to non-Euclidean settings such as manifolds and graphs.  Given the increasing importance over the past few decades of graphs and networks as models for systems generating and measurable via "Big Data," the theory of Spectral Graph Wavelets in \cite{MR2754772} provides such an extension to the setting of graphs.  The spectral approach taken in \cite{MR2754772}, based on analogy with the Euclidean theory and using the spectral decomposition of the graph Laplacian, is natural and in our view the appropriate framework to take.  However, the theorems that guarantee some of the key properties that makes wavelets a useful and general tool in the classical setting have not, to our knowledge, been established in the same generality in the setting of general networks.  Our goal in this note is to present a particular class of wavelets, what we call Hermitian Graph Wavelets, and establish some of these key properties, within the general framework of \cite{MR2754772}.

One of the features of primary importance in classical wavelet analysis is the ability of the wavelet kernel to localize in frequency as well as in space.  In the classical setting of the real axis it is the action of the scaling group $\R^*=(0,\infty)$ on $\R$ that makes the construction of such wavelets possible: if a wavelet kernel function $\psi(x):\R\to\R$ has support approximately contained in an interval of radius $r$, then $\psi(sx)$ has support approximately contained in a interval of radius $\frac rs$.  Moreover in the classical setting we have the Fourier transform and its well-known transformation under rescaling, $\widehat{f(sx)}(\xi)=\frac1s\hat f(\frac\xi s)$.  In the setting of a graph where no such obvious action of $\R^*$ exists, things are not so straightforward.  In the graph setting the Laplacian and its spectral decomposition provide an analogue of the Fourier transform, but how should one define scale on graph and what is the relationship of such a notion to the spectrum of the Laplacian?  To address these questions, we take inspiration from history.  

Wavelets have their roots in classical Littlewood-Paley theory, and there one can find what can be considered possibly the first proto wavelet, the derivative of the Poisson kernel (see \cite{MR1107300} for a good historical account of this development).  The theory of Spectral geometry and related manifold learning techniques (e.g$.$ \cite{MR2238665}) makes clear that a good strategy for multiscale analysis on any space with a well defined Laplacian is to study the associated heat kernel.  These facts provide a strong suggestion that one should mimic the classical development with the heat kernel replacing the Poisson kernel, which is the approach we take here.  The heat kernel is defined as the integral kernel of the heat semigroup, $e^{-t\Delta}$:
\[H_t(x,y)=\sum_{k=0}^Ne^{-t\lambda_k}\phi_k(x)\phi_k(y).\]
The behavior of $H_t(x,y)$, in particular its decay, is well known to encode the geometric structure of a network and by defining our wavelet kernel in terms of $H_t$ we can transfer much of this geometric content.  In particular, sharp localization bounds for $H_t$ have recently been established, and these allow us to define wavelet satisfying similar bounds.  The intimate connection between $H_t$ and the geometry of the underlying graph is also inherited, as we will attempt to illustrate.

In the next section we cover some preliminaries and fix notation, review the framework of \cite{MR2754772}, and cover some important properties of the heat kernel before turning to the definition and main properties of Hermitian graph wavelets.  After establishing our main theorem on localization, we proceed to make a connection with the so-called Information Centrality, a measure of node importance in networks recently shown to be essential in the leader detection problem for a system of interacting agents.  We then describe some potential applications to quantifying leadership hierarchies in general networks of interacting agents.

\section{Preliminaries}
Throughout we will denote by $G$ an undirected, weighted graph with edge weights $w_{x,y}$ between vertices $x$ and $y$.  The adjacency matrix is defined by $A_{i,j}=w_{i,j}$ and the Laplacian as $\Delta=D-A$ where $D=A\cdot1$.  We let $N$ be the number of vertices in $G$ and denote by $L^2(G)$ the space of vectors $u,v:G\to \R$ with inner product \[\<u,v\>=\sum_{x\in G} u(x)v(x).\] $L^2(G)$ is obviously isomorphic to $\R^{N}$ with the usual Euclidean structure.  $\Delta$ acts on functions $f\in L^2(G)$ as
\[\Delta f(x)=\sum_{z\sim x}w_{x,z}(f(x)-f(z))\] 
in direct analogy with the classical Laplacian.

We denote by $\{\lambda_k\}_{k=0}^N$ and $\{\phi_k\}_{k=0}^N$ respectively the eigenvalues and eigenvectors of $\Delta$.  We have $\lambda_0=0$ and $\lambda_k\leq\lambda_{k+1}$.  We also have $\phi_0=1/\sqrt{N}$.

\subsection{Spectral Graph Wavelets}
The theory set out in \cite{MR2754772} is based on the Laplacian as follows: a family of wavelets is defined as the kernels of a one parameter family of operators $g(s\Delta)$ for some $g$ satisfying $g(0)=0$ and $\int_0^\infty \frac{g^2(x)}x\,dx<\infty$:\[\psi_{s,x}(y)=\sum_{k=1}^\infty g(s\lambda_k)\phi_k(x)\phi_k(y).\]

The authors prove that given such a $g$ (in that paper they construct an example using splines), one has analogues of many useful properties from the classical setting: one can define a continuous transform and prove an inversion formula, by selecting a discrete number of values $\{s_n\}$ one obtains a frame with frame bounds depending on $g$ and a low pass filter, etc.  

Regarding localization, the authors give the following theorem: There exists constants $D$ and $t_0$ such that \[\frac{\psi_{s,x}(y)}{\|\psi_{s,x}\|}\leq Ds\] for all $s<t_0$ and where $D$ and $t_0$ depend on the number of edges in the shortest path from $x$ to $y$.  The authors follow the theorem with the following remark: \textit{``As this localization result uses the shortest path distance defined without using edge weights, it is only directly useful for sparse weighted graphs where a significant number of edge weights are exactly zero.  Many large scale graphs which arise in practice are sparse, however, so the class of sparse weighted graphs is of practical significance."}  The authors are completely correct in the last statement, however there are many important applications where such sparsity does not hold, e.g., mobile communication networks (an application of spectral graph wavelets to such systems was a motivation for the current work).  Moreover, one may wonder if bounds of a Gaussian or sub-Gaussian nature analogous to those in other settings holds in the case of graphs.  We will answer this question below.

\subsection{Heat Kernel}
$\Delta$ is symmetric and nonnegative and thus by the spectral theorem we can define for $t>0$ the heat kernel $p_t(x,y)$ as 
\[e^{-t\Delta}=\sum_{k=0}^N e^{-t\lambda_k}\phi_k(x)\phi_k(y)=H_t(x,y).\]
One of the most important features of the heat kernel is that its decay reflects the geometry of $G$, as is intuitively clear from the heat equation, and as is made precise by the theory of \cite{MR3673659}, as follows:  First, suppose we are given a metric on $G$, that is a map $\rho:V\times V\to[0,\infty)$ that satisfies the triangle inequality.  Such a metric is called \textit{intrinsic} if for each $x\in V$ we have $\sum_{y\sim x}w_{x,y}\rho^2(x,y)<1$.  Such a metric always exists, as one can take 
\[\rho(x,y)=\frac{N}{\sqrt{w_{x,y}}}\]
where 
\[N=\max_{x\in G} \,\,deg_x\] is the maximum degree taken over all $x\in G.$
Given an intrinsic metric, we define its \textit{jump size} to be $\sup_{x\sim y}\{\rho(x,y)\}$, meaning the supremum over all pairs of points with edges with nonzero weight.  With these definitions in hand we now state the main result of \cite{MR3673659}:  If we let \[\zeta_s(t,r)=\frac1{s^2}\left(rs \cdot arcsinh\frac{rs}t -\sqrt{t^2+r^2s^2} + t\right),\] then we have \begin{equation} H_t(x,y)\leq e^{-\zeta_s(t,\rho(x,y))}\end{equation} where $s$ is the jump size of the intrinsic metric $\rho$.  Note the bounds dependence on the choice of intrinsic metric. 

\section{Hermitian Graph Wavelets}
By analogy with the classical theory, we define our wavelet defining function as the time derivative of the heat kernel \[g(\Delta)=\frac d{dt}e^{-t\Delta}|_{t=1}=\Delta e^{-\Delta}\]
where we have removed the factor of $-1$ for convenience.  The full kernel is then the kernel of the operator $g(s\Delta)=s\Delta e^{-s\Delta}$, 
\[\psi_{s,x}(y)=\sum_{k=1}^Ns\lambda_k e^{-s\lambda_k}\phi_k(x)\phi_x(y).\] 

We see that $g$ satisfies the properties required for the basic properties of Spectral Graph Wavelets to hold, i.e., continuous inversion, defining a frame, etc.  What will show next is that the relationship to the heat kernel allows us to obtain much sharper localization than previously available.

\subsection{Localization}

Following the method of proof in \cite{MR1023321} we obtain the following:
\begin{thm} With the notation above, we have the following bound:
\[|\psi_{t,x}(y)|\leq \left[\frac{r^2}t\left(1+\frac{s}{\sqrt{t^2+s^2r^2}}\right)\left(\frac1{sr+\sqrt{t^2+s^2r^2}}\right)-\left(\frac{t}{\sqrt{t^2+s^2r^2}}+1\right)+\frac ct\right]e^{-\zeta_s(t,r)}\]
where $c>0$ is a constant depending on $G$.
\end{thm}
\begin{proof}
Fix $x$ and $y$ and let number $b=\rho(x,y)\geq0$ for some intrinsic metric $\rho$ and note that from the definitions above we may consider the time/scale parameter to be complex valued.  Thus let $f(z)=H_z(x,y)$ for $Re(z)\geq0$.  
Now we let \[h(z)=f\left(\frac1z\right)e^{\zeta_s\left(\frac1z,b\right)}\]
and as in \cite{MR1023321} we will apply the Phragm\'{e}n-Lindelh\"{o}f principle (see e.g. \cite{MR1976398}) to bound $h$.  To that end, we first restrict attention to the sector $\{z=re^{i\theta}\,:\,\theta\in [0,\frac\pi2]\}$.  For $\theta=0$, we have from eq $(1)$ that $|h(r)|< 1$ for $r>0$.  

For $\theta=\frac\pi2$ we have 
\[\zeta_s\left(\frac1{iy},b\right)=\frac1{s^2}\left(bs\log(-iybs+\sqrt{1-(ybs)^2})-\sqrt{(bs)^2-\frac1{y^2}}-\frac iy\right),\]
so that \begin{align}\notag |e^{\zeta_s\left(\frac1{iy},b\right)}|&\leq|e^{\frac bs\log(-iybs+\sqrt{1-(ybs)^2})}||e^{-\frac1{s^2}\sqrt{(bs)^2-\frac1{y^2}}}|\\\notag&\leq|e^{\frac bs\log(\sqrt{(ybs)^2+1-(ybs)^2})}||e^{-\frac1{s^2}\sqrt{(bs)^2-\frac1{y^2}}}|\\\notag&\leq|e^{-\frac1{s^2}\sqrt{(bs)^2-\frac1{y^2}}}|.\end{align}
If $y=\frac1{bs}$ then $\sqrt{(bs)^2-\frac1{y^2}}=0$, if $y<\frac1{bs}$ then $\sqrt{(bs)^2-\frac1{y^2}}$ is purely imaginary, and if $y>\frac1{bs}$ then $\sqrt{(bs)^2-\frac1{y^2}}>0$, so that in any case $|e^{\zeta_s\left(\frac1{iy},b\right)}|\leq1$.  We also have that $\left|f\left(\frac 1{iy}\right)\right|\leq\sum|\phi_k(x)\phi_k(y)|\equiv a>0$, so that $|h(iy)|\leq a$.  

Next note that there exist $C,c>0$ such that \[|e^{\zeta_s\left(\frac1z,b\right)}|\leq |e^{\frac bs \cdot arcsinh\frac{bs}z}||e^{\frac1{s^2}\left(\frac1z-\sqrt{\frac1{z^2}+(bs)^2}\right)}|\leq Ce^{c|z|}.\]
Therefor by the Phragm\'{e}n-Lindelh\"{o}f principle we have that \[|h(z)|\leq\max(1,a)\] for all $z$ in the sector $\{z=re^{i\theta}\,:\,\theta\in [0,\frac\pi2]\}$.  The same arguments hold for the sector $\{z=re^{i\theta}\,:\,\theta\in [-\frac\pi2,0]\}$, so that we have $|h(z)|\leq\max(1,a)\equiv c$ for all $z$ with $Re(z)>0$ and therefor also that 
\[\left|F\left(z\right)\right|\leq\max(1,a)\equiv c\] with $F(z)=h(\frac1z)$.  We then consider the circle of radius $\alpha t$ for some $\alpha\in(0,1)$ centered at $t>0$ and apply the Cauchy integral formula to obtain $|F^\prime(t)|\leq\frac c{\alpha t}$.  Since this holds for any $\alpha\in(0,1)$ we have 
\[|F^\prime(t)|\leq\frac c{ t},\] 
\[\left|\left(\frac{d}{dt}\zeta_s(t,b) \right)e^{\zeta_s(t,b)}f(t)+f^\prime(t)e^{\zeta_s(t,b)}\right|\leq\frac ct,\]
and therefor that
\[|f^\prime(t)|\leq \left|\frac{d}{dt}\zeta_s(t,b) \right||f(t)|+\frac ct e^{-\zeta_s(t,b)}\]
which together with eq$.$ $(1)$ implies the conclusion of the Theorem.

\end{proof}
This bound, which to our knowledge has not appeared in the literature before, gives precise quantitative information on the decay of $|\psi_{t,x}(y)|$ as $\rho(x,y)\to\infty$ and ensures sharp localization of wavelets and holds for an arbitrary weighted graph and intrinsic metric.  Note that the proof depends essentially upon the relation to $H_t(x,y)$, which was a primary motivation for our choice of $g$.  It should also be noted that this result is clearly not optimal for $t$ near zero and $x=y$ due to the introduction of the singular term $\frac ct$.  However, for $x\neq y$ the result still gives strong bounds on $|\psi_{t,x}(y)|$ for $\rho(x,y)\gg t$, which is precisely the regime we are concerned with when speaking of localization of wavelet functions.


\section{Mean Diffusion Time, Information Centrality and Leadership Quantification}
\subsection{Mean Diffusion time} 
The value of $\|\psi_{t,x}\|^2$ measures the total energy of the derivative of $H_t(x,\cdot)$ over the network.  The function $t\mapsto\|\psi_{t,x}\|^2$ can be viewed as a (non-normalized) density on the set of possible scales, and the larger values signify scales where heat is flowing strongly, i.e., at which the vertex $x$ is most strongly influencing the rest of the network.  For each vertex $x$ we define the \textit{mean diffusion time, $MDT(x)$,} as the average value of this density:
\[MDT(x)=\int_0^\infty \|\psi_{t,x}\|^2 \,dt.\]
We can rewrite this integral as $\int_0^\infty t\|\psi_{t,x}\|^2 \,\frac{dt}t$ and in this way view $MDT(x)$ as a multiplicative average over all scales. In order to see that $MDT(x)$ always exists, note that \[\|\psi_{t,x}\|^2=\sum_{k=1}^N t^2\lambda_k^2e^{-2t\lambda_k}|\phi_k(x)|^2=O\left(\lambda_1^2e^{-2t\lambda_1}\right),\] where as always $\lambda_1>0$.  $MDT(x)$ thus gives a new measure of node centrality on a network $G$, where the smaller value of $MDT(x)$ indicates greater node importance, as it takes less time for heat diffusing from the vertex $x$ to reach the rest of the network.

\subsection{Information Centrality, leader detection} 
A problem of central importance in the study of cooperative systems of interacting agents (e.g$.$ biological systems such as schools of fish) is the quantification of the hierarchical leadership structure of the group using a general, data driven methodology.  Often, measures of graph centrality are used.  In \cite{6761082}, the authors prove that for a system of agents in a stochastic environment attempting to collectively measure and respond to an ambient signal, the optimal choice of a single agent for measuring the ambient signal and communicating that information to the rest of the network (the ``leader") is determined by maximizing so called ``information centrality" (IC).  This result strongly suggests that IC should be utilized in analyses of hierarchical leadership structure.  

Our second result of the paper is the following
\begin{thm} For any graph $G$ we have
\[\argmin_{x\in G} MDT(x)=\argmax_{x\in G}IC(x)\]
\end{thm}
\begin{proof}
First we cite the following result of \cite{strathprints30110} (eq 2.25): 
\[IC(x)=\left(\sum_{k=1}^N\frac{|\phi_k(x)|^2}{\lambda_k} + \frac1N\sum_{x}\left(\sum_{k=1}^N\frac{|\phi_k(x)|^2}{\lambda_k}\right)\right)^{-1}.\]
The second term on the right is independent of $x$, and therefor we have 
\[\argmax_x IC(x)=\argmin_x \frac{|\phi_k(x)|^2}{\lambda_k}.\]
A straightforward calculation gives the following equality: \[\int_0^\infty \|\psi_{s,x}\|^2ds=\sum_{k=1}^N\int_0^\infty s^2\lambda_k^2e^{-2s\lambda_k}ds|\phi_k(x)|^2=C\sum_{k=1}^N\frac{|\phi_k(x)|^2}{\lambda_k}.\] 
From here the statement of the theorem easily follows.
\end{proof}

Thus we obtain a characterization of leadership in terms of $MDT(x)$, which matches our intuition: the optimal leader to communicate the ambient signal to the rest of the group is the one who communicates most efficiently with the group, i.e., the agent with smallest mean diffusion time.  $MDT(x)$ thus solves the leadership selection problem and has a clear intuitive meaning to boot.  

\section{Acknowledgements}
The authors would like to acknowledge support from the DARPA YFA project N66001-17-1-4038.

\bibliographystyle{plain}
\bibliography{HGW}
\end{document}